\documentclass[12pt]{article}
\usepackage{booktabs}
\usepackage{caption}
\usepackage{mathrsfs}
\usepackage{amsmath}
\usepackage{amsfonts,amsthm,amssymb,mathrsfs,bbding}
\usepackage{graphics,multicol}
\usepackage{graphicx}
\usepackage{color}
\usepackage{enumerate}
\usepackage{caption}
\usepackage{rotating}
\usepackage{lscape}
\usepackage{longtable}
\allowdisplaybreaks[4]
\usepackage{tabularx}
\usepackage[colorlinks=true,anchorcolor=blue,filecolor=blue,linkcolor=blue,urlcolor=blue,citecolor=blue]{hyperref}
\usepackage{extarrows}
\usepackage{cite}
\usepackage{latexsym,bm}
\usepackage{mathtools}
\evensidemargin=\oddsidemargin\topmargin=-1.5cm

\newtheorem{thm}{Theorem}

\newtheorem{lem}{Lemma}
\newtheorem{cor}{Corollary}

\newtheorem{claim}{Claim}
\theoremstyle{definition}

\addtocounter{section}{0}
\begin{document}

	\title{\bf Spectral condition for spanning $k$-ended trees in  $t$-connected graphs}
	\author{{Jiaxin Zheng,  Xueyi Huang\footnote{Corresponding author.}\setcounter{footnote}{-1}\footnote{\emph{E-mail address:} huangxymath@163.com}, \ Junjie Wang}\\[2mm]
		\small School of Mathematics, East China University of Science and Technology,\\
		\small Shanghai 200237, China}

	\date{}
	\maketitle
	{\flushleft\large\bf Abstract}  For any integer $k\geq 2$, a spanning $k$-ended tree is a spanning tree with at most $k$ leaves. In this paper, we provide a tight spectral radius condition for the existence of a spanning $k$-ended tree in  $t$-connected graphs, which generalizes a result of Ao, Liu and Yuan (2023).
	
	\begin{flushleft}
	\end{flushleft}

	\section{Introduction}
	All graphs considered in this paper are undirected and simple. Let $G$ be a graph with vertex set $V(G)$ and  edge set $E(G)$, and let $e(G)=|E(G)|$. For any $v\in V(G)$, we denote by $N_G(v)$ and $d_G(v)$ (or $N(v)$ and $d(v)$ for short) the \textit{neighborhood} and \textit{degree} of $v$, respectively. For any $S\subseteq V(G)$, let $N_G(S)$ (or $N(S)$ for short) denote the set of vertices in $G$ that is adjacent to some vertex of $S$, and  let $G[S]$ denote the subgraph of $G$ induced by $S$. A \textit{clique} (resp.,  \textit{independent set}) of $G$ is a subset $S$ of $V$ such that $G[S]$ is complete (resp., empty). The \textit{clique number} (resp., \textit{independent number}) of $G$, denoted by $\omega(G)$ (resp., $\alpha(G)$), is the number of vertices in a maximum clique (resp., independent set) of $G$. A connected graph $G$ is called $t$-\textit{connected} if it has more than $t$ vertices and remains connected whenever fewer than $t$ vertices are removed. The \textit{join} of two graphs $G_1$ and $G_2$, denoted by $G_1\nabla G_2$,  is the graph obtained from $G_1\cup G_2$ by adding all possible edges between $G_1$ and $G_2$. Also, let $K_n$ denote the complete graph of order $n$.


A \textit{spanning tree} $T$ of a connected graph $G$ is a subgraph of $G$ that is a tree which includes all of the vertices of $G$.  For any integer $k \ge 2$, a spanning tree with at most $k$ leaves is called a  \textit{spanning $k$-ended tree}, where a leaf is a vertex of degree one. In particular, a spanning $2$-ended tree is called a  \textit{Hamilton path}.

In 1960,  Ore \cite{Or} presented a famous degree sum condition for a connected graph to have a Hamilton path.

\begin{thm}(Ore \cite{Or})\label{thm::1}
Let $G$ be a connected graph of order $n$. If $d(u) + d(v) \ge n-1$ for every two non-adjacent vertices $u$ and $v$, then $G$ has a Hamilton path.
\end{thm}


In 1972, Chv\'{a}tal and Erd\H{o}s \cite{CE} proved that every $t$-connected graph with independent number at most $t+1$ has a  Hamilton path.

\begin{thm}(Chv\'{a}tal and Erd\H{o}s \cite{CE})\label{thm::2}
Let $t \ge 1$ be an integer, and let $G$ be a $t$-connected graph. If $\alpha (G) \le t+1$, then $G$ has a Hamilton path.
\end{thm}

In \cite{BT}, Broersma and Tuinstra generalized the conclusion of Theorem \ref{thm::1} by providing a degree sum condition for the existence of a spanning $k$-ended tree in connected graphs.

\begin{thm}(Broersma and Tuinstra \cite{BT})\label{thm::3}
Let $G$ be a connected graph of order $n$, and let $k \ge 2$ be an integer. If $d(u)+d(v) \ge n-k+1$ for every two nonadjacent vertices $u$ and $v$, then $G$ has a spanning $k$-ended tree.
\end{thm}

In \cite{Win}, Win generalized the result of Theorem \ref{thm::2} to spanning $k$-ended trees, which confirms a conjecture of Las Vergnas.

\begin{thm}(Win \cite{Win})\label{thm::4}
Let $k \ge 2$ be an integer,  and let $G$ be a $t$-connected graph. If $\alpha (G) \le k+t-1$, then $G$ has a spanning $k$-ended tree.
\end{thm}

For any fixed integer $l \ge 0$, the $l$-\textit{closure} of $G$, denoted by $C_{l}(G)$,  is the graph obtained from $G$ by recursively joining pairs of non-adjacent vertices whose degree sum is at least $l$ until no such pair remains. Also, we say that $G$ is $l$-\textit{closed} if $C_{l}(G)=G$.   In \cite{BT}, Broersma and Tuinstra presented a closure theorem for spanning $k$-ended tree.

\begin{thm}(Broersma and Tuinstra \cite{BT})\label{thm::5}
Let $G$ be a connected graph of order $n$, and let $k$ be an integer with $2 \le k \le n-1$. Then $G$ has a spanning $k$-ended tree if and only if the $(n-1)$-closure $C_{n-1}(G)$ of $G$ has a spanning $k$-ended tree.
\end{thm}

The \textit{adjacency matrix} of a graph $G$ is defined as $A(G)=(a_{u,v})_{u,v\in V(G)}$, where $a_{u,v}=1$ if $u$ and $v$ are adjacent in $G$, and $a_{u,v}=0$ otherwise. The (adjacency) \textit{spectral radius} of $G$, denoted by $\rho(G)$, is the largest eigenvalue of  $A(G)$. In recent years, the problem of finding spectral conditions for graphs having certain structural properties or containing specified kinds of subgraphs has received considerable attention. Cioab\u{a}, Gregory and Haemers \cite{CGH} found a best upper bound on the third largest eigenvalue that is sufficient to guarantee that an $n$-vertex $k$-regular graph $G$ has a perfect matching when $n$ is even, and a matching of order $n-1$ when $n$ is odd. Fiedler and Nikiforov \cite{FN} gave a spectral radius condition for graphs to have a Hamilton cycle or Hamilton path. Li and Ning \cite{LN} provided a tight spectral radius condition for graphs with bounded minimum degree to have a Hamilton cycle or Hamilton path. 
Cioab\u{a}, Feng, Tait and Zhang \cite{CFTZ} provided a tight spectral radius for graphs to contain a friendship graph of given order as a subgraph. For more results on this topic, we refer the reader to  \cite{ALY,BZ,O16,O22,OC,OPPZ,WKX,ZL}, and references therein.

In this paper, we study on the spectral conditions of spanning $k$-ended trees in $t$-connected graphs. To achieve this goal, we first present an edge number condition.

\begin{thm}\label{thm::6}
Let $k\geq 2$ and $t\geq 1$ be integers, and let  $G$ be a $t$-connected graph of order $n$.  If $n \ge \max\{ 6k+6t-1, k^2+t k+t+1\}$ and
$$e(G) \ge {n-k-t \choose 2}+(k+t-1)^2+k+t,$$ then $G$ has a spanning $k$-ended tree unless $C_{n-1}(G) \cong K_t \nabla (K_{n-k-2t+1} \cup (k+t-1)K_{1})$.
	\end{thm}

Based on Theorem \ref{thm::6}, we provide the following tight spectral radius condition for the existence of a spanning $k$-ended tree in $t$-connected graphs. 
\begin{thm}\label{thm::7}
Let $k\geq 2$ and $t\geq 1$ be integers, and let  $G$ be a $t$-connected graph of order $n$. If $n \ge \max\{6k+6t-1, k^2+\frac{3}{2}kt+\frac{1}{2}t^2+\frac{1}{2}t+1\}$ and 
$$
\rho(G) \ge \rho(K_t \nabla (K_{n-k-2t+1} \cup (k+t-1)K_1)),
$$ 
then $G$ has a spanning $k$-ended tree unless $G \cong K_t \nabla (K_{n-k-2t+1} \cup (k+t-1)K_1)$.
\end{thm}

By Theorem \ref{thm::6}, we immediately deduce a result of  Ao, Liu and Yuan regarding the existence of a spanning $k$-ended tree in connected graphs\cite{ALY}.

\begin{cor}(Ao, Liu and Yuan \cite{ALY})\label{cor::7}
Let $G$ be a connected graph of order $n$ and $k \ge 2$ be an integer. If $n \ge \max \{ 6k+5, k^2+\frac{3}{2}k+2 \}$ and 
$$
\rho (G) \ge \rho (K_1 \nabla (K_{n-k-1} \cup kK_1)),  
$$
then $G$ has a spanning $k$-ended tree unless $G \cong K_1 \nabla (K_{n-k-1} \cup kK_1)$.
\end{cor}



	\section{Proof of the main results}

In this section, we will prove Theorem \ref{thm::6}, and use it to deduce Theorem \ref{thm::7} directly. To achieve this goal, we need some useful lemmas.

Let $T$ be a tree. An \textit{internal vertex} of $T$ is a vertex with degree at least two. The following lemma suggests that the graph $K_t \nabla (K_{n-k-2t+1} \cup (k+t-1)K_1)$ considered in Theorems \ref{thm::6} and \ref{thm::7} contains no spanning $k$-ended trees.

\begin{lem}\label{lem::1}
Let $k \ge 2$ and $t\ge 1$ be integers. Then the graph $K_t \nabla (K_{n-k-2t+1} \cup (k+t-1)K_1)$ does not have a spanning $k$-ended tree.   
    \end{lem}
    \begin{proof}
        Let $V_1=V(K_t)$ and $V_2=V((k+t-1)K_1)$. Suppose to the contrary that $T$ is a spanning $k$-ended tree of $G$.  Then $V_2$ has at least $t$ internal vertices of $T$, say $v_1,\ldots,v_t$. Let $V_2'=\{ v_1, \ldots, v_t \}$, and let $T'=T[V_1 \cup V_2']$ denote the subgraph of $T$ induced by $V_1\cup V_2'$. Clearly, $T'$ is a forest of order $2t$.  Since $N_T(v_i)\subseteq V_1\subseteq V(T')$, we have $d_{T'}(v_i) \ge 2$ for $1 \le i \le t$, and hence $e(T')\geq \sum_{i=1}^{t}d_{T'}(v_i) \ge 2 t=|V(T')|$. Therefore, we may conclude that $T'$ contains a cycle, which is impossible.
    \end{proof}

        \begin{lem}\label{lem::2}
Let $k \ge 2$ and $t\ge 1$ be integers. Suppose that $G$ is an $(n-1)$-closed graph of order $n$ with $n\ge 6k+6t-1$. If 
            \begin{equation}
              e(G) \ge  {n-k-t \choose 2}+(k+t-1)^2+k+t,
              \nonumber
            \end{equation}
            then $\omega (G) \ge n-k-t+1$.
	\end{lem}

    \begin{proof}
        Since $G$ is $(n-1)$-closed, any two vertices of degree at least $\frac{n-1}{2}$ must be adjacent in $G$. Let $C$ be the vertex set of a maximum clique in $G$ containing all vertices of degree at least $\frac{n-1}{2}$, and let $H$ be the subgraph of $G$ induced by $V(G) \backslash C$. Let $|C|=r$. We consider the following two situations.
   
{\flushleft\textit{Case 1.}} $1 \le r \le \frac{n}{3}+k+t$.   
        
Note that $d_{C}(u) \le r-1$ and $d_{G}(u) \le \frac{n-2}{2}$ for all $u \in V(H)$. We have 
	  \begin{equation*}
	  \begin{aligned}
            e(G)&=e(G[C])+e(H)+e(V(H),C)\\
          &={r \choose 2}+\frac{1}{2}\left(\sum_{u\in V(H)} d_{G}(u)-\sum_{u\in V(H)} d_{C}(u)\right)+\sum_{u \in V(H)} d_{C}(u)\\
            &={r \choose 2}+\frac{1}{2}\left(\sum_{u\in V(H)} d_{G}(u)+\sum_{u\in V(H)} d_{C}(u)\right)\\
            &\le {r \choose 2}+\frac{1}{2}\left((n-r)\cdot \frac{n-2}{2}+(n-r)(r-1)\right)\\
            &=\frac{n+2}{4}\cdot r+\frac{n^2}{4}-n\\
            &\le \frac{n+2}{4}\cdot  \left(\frac{n}{3}+k+t\right)+\frac{n^2}{4}-n\\
            &=\frac{4n^2+(3k+3t-10)n+6(k+t)}{12}\\
            &={n-k-t \choose 2}+(k+t-1)^2+k+t\\
            &~~~~-\frac{(n-6k-6t+1)(2n-3k-3t+2)+3(k+t)+10}{12}\\
            &\leq  {n-k-t \choose 2}+(k+t-1)^2+k+t-\frac{3(k+t)+10}{12}~~\mbox{(as $n \ge 6k+6t -1$)}\\
            &<e(G),
        \end{aligned}
	  \end{equation*}
        which is a contradiction.

{\flushleft \textit{Case 2.}} $\frac{n}{3}+k+t <r \le n-k-t$.   

Note that, for each  $u \in V(H)$, $d_{G}(u) \le n-r-1$, since otherwise $u$ would be adjacent to all vertices of $C$ because $G$ is $(n-1)$-closed, contrary to the maximality of $C$. Then we have
        \begin{equation*}
	  \begin{aligned}
             e(G)&=e(G[C])+e(C,V(H))+e(H)\\
             &={r \choose 2}+\frac{1}{2}\left(\sum_{u\in V(H)} d_{G}(u)+\sum_{u\in V(H)} d_{C}(u)\right)\\
             &\le {r\choose 2}+\sum_{u \in V(H)}d_G(u)\\
             &\le {r\choose 2}+(n-r)(n-r-1)\\
             &={n-k-t \choose 2}+(k+t-1)^2+k+t\\
             &~~~~-\frac{(n-k-t-r)(3r-n-3k-3t+1)}{2}-1\\
             &\le {n-k-t \choose 2}+(k+t-1)^2+k+t-1~~\mbox{(as $\frac{n}{3}+k+t <r \le n-k-t$)}\\
             &<e(G),
        \end{aligned}
	  \end{equation*}
       which is also impossible.

  Therefore, we conclude that $\omega (G) \ge |C|=r\geq n-k-t+1$, and the result follows.        
    \end{proof}

    A \textit{matching} $M$ of $G$ is a set of edges that do not have a set of common vertices.  We say that a matching $M$ \textit{saturates} a vertex $v$ in $G$, or $v$ is $M$-\textit{saturated} in $G$, if some edge of $M$ is incident with $v$. 
    

    \begin{lem}(Hall \cite{Ha})\label{lem::3}
		Let $G$ be a bipartite graph with bipartition $(X, Y)$. Then $G$ contains a matching that saturates every vertex in $X$ if and only if
  $$
  |N(S)| \ge |S|
  $$ 
  for all $S \subseteq X$.
	\end{lem}

Now we are in a position to give the proof of Theorem \ref{thm::6}.


{\flushleft \it Proof of Theorem \ref{thm::6}.}  
Suppose that $G$ does not have a  spanning $k$-ended tree.  Let $G'=C_{n-1}(G)$ be the $(n-1)$-closure of $G$. By Theorem \ref{thm::5}, $G'$ contains no spanning $k$-ended trees. As $G$ is $t$-connected,  $G'$ is also $t$-connected,  and 
        $$e(G') \ge e(G) \ge {n-k-t \choose 2}+(k+t-1)^2+k+t.$$ 
        By Lemma \ref{lem::2}, we have  $\omega (G') \ge n-k-t+1$. Let $C$ be a maximum clique of $G'$, and let $H$ be the subgraph of $G'$ induced by $V(G') \backslash C$. We have the following two claims.

        \begin{claim}\label{claim::1}
            $\omega (G') =|C|=n-k-t+1$,  and $V(H)$ is an independent set.
        \end{claim}
        
{\flushleft \it Proof of Claim \ref{claim::1}.}  
 Since $G'$ has no spanning $k$-ended trees, by Theorem \ref{thm::4}, we have $\alpha(G') \ge t+k$, and hence $\omega (G') \le n-k-t+1$. Combining this with $\omega (G') \ge n-k-t+1$, we obtain $\omega (G')=|C|=n-k-t+1$, and it follows that  $V(H)$ is an independent set because $\alpha(G')\geq t+k$. \qed\vspace{3mm}

According to Claim \ref{claim::1}, the vertex set of $G'$ can be partitioned as $V(G')=C\cup V(H)$, where $C$ is a clique of size $n-k-t+1$, and $V(H)$ is an independent set of size $k+t-1$. Moreover, for  any non-empty subset $U$ of $V(H)$, $N(U)\subseteq C$, and $|N(U)|\geq t$ because $G'$ is $t$-connected. To determine the structure of $G'$, it suffices to analyse the edges between $C$ and $V(H)$. 

        \begin{claim}\label{claim::2}
            $|N(V(H))|=t$, and consequently, $G' \cong K_t \nabla (K_{n-k-2t+1} \cup (k+t-1)K_1)$.
        \end{claim}
{\flushleft \it Proof of Claim \ref{claim::2}.}  
Let $X=\{v_1,\ldots,v_{t+1}\}$ be a subset of $V(H)$ such that 
$$|N(X)|=\max\{|N(U)|: U\subseteq V(H)~\mbox{and}~|U|=t+1\}.$$ 
To prove the claim, it remains to show that $|N(X)|=t$. By contradiction, assume that $|N(X)|\geq t+1$. We shall prove that $G'[C\cup X]$ has a Hamilton path. Let $H'$ denote the bipartite subgraph of $G'$ with bipartition $X\cup N(X)$ induced by the edges between $X$ and $N(X)$. For any non-empty subset $S$ of $X$, if $|S|\leq t$, then $N(S)\geq t\geq |S|$ by above arguments, and if $|S|=t+1$, i.e., $S=X$, then $|N(S)|=|N(X)|\geq t+1=|S|$. By Lemma \ref{lem::3}, there exists a matching in $H'$ (and so in $G'$), say $M=\{v_1u_1, v_2u_2, \ldots, v_{t+1}u_{t+1}\}$ ($u_i\in N(X)\subseteq C$ for $1\leq i\leq t+1$), that saturates every vertex in $X$. Let $Y=\{u_1, u_2, \ldots, u_{t+1}\}$ and $Z=C\setminus Y=\{u_{t+2}, \ldots, u_{n-k-t+1}\}$.  Let $P_1,\ldots,P_s$ be a sequence of vertex-disjoint paths in $H'$ that starts from a vertex of $Y$ and ends at a vertex of $X$ with edges occurring alternately in $M$ and $E(H')\setminus M$ such that the following three conditions hold:
\begin{enumerate}[(i)]
    \item $X\cup Y=V(P_1)\cup\cdots\cup V(P_s)$;
    \item $P_{1}$ is maximal in $H'$;
    \item if $s\geq 2$, then $P_{i}$ is maximal in $H'-\cup_{j=1}^{i-1} P_{j}$ for $2\leq i\leq s$.
\end{enumerate}
Clearly, $1\leq s\leq t+1$. If $s=1$, then $P_1$ is a Hamilton path in $G'[X\cup Y]$ that starts from a vertex of $Y$ and ends at a vertex of $X$, and so $G'[C\cup X]$ also contains a Hamilton path because $C$ is a clique. Now suppose that $s\geq 2$. Without loss of generality, we denote 
$$
\begin{aligned}
P_1&=u_1v_1 \cdots u_{i_1}v_{i_1},\\
P_2&=u_{i_{1}+1}v_{i_{1}+1} \cdots u_{i_2}v_{i_2},\\
&~~~~~~~~~~~~~\vdots\\
P_s&=u_{i_{s-1}+1}v_{i_{s-1}+1} \cdots u_{i_s}v_{i_s},
\end{aligned}
$$
where $1\leq i_1<i_2<\ldots<i_{s-1}<i_s=t+1$.
Recall that  every vertex of $X\subseteq V(G')$ has at least $t$ neighbors in $N(X)\subseteq C$. According to  (i)--(iii), we assert that $v_{i_j}$ has at least $t-i_j$ neighbors in $N(X)\setminus Y\subseteq Z$ for each $j\in \{1,\ldots,s-1\}$. Also note that $0\leq t-i_{s-1} < t-i_{s-2} < \cdots < t-i_{1}$. If $s=2$, then $P_1=u_1v_1\cdots u_{i_1}v_{i_1}$ and $P_2=u_{i_1+1}v_{i_1+1}\cdots u_{t+1}v_{t+1}$. As $C$ is a clique, it is easy to see that $v_{t+1}P_2u_{i_1+1} u_{t+2}\cdots u_{n-k-t+1} u_1P_1v_{i_1}$ is a Hamilton path in $G'[C\cup X]$ that starts from $v_{t+1}$ and ends at $v_{i_1}$. If $s\geq 3$, then it follows from the inequality  $0\leq t-i_{s-1} < t-i_{s-2} < \cdots < t-i_{1}$ that  $$|N_Z(v_{i_j})|\geq t-i_{j}\geq s-j-1,$$
where $1\leq j\leq s-2$. Hence, there are  $s-2$ distinct vertices in $Z$, say  $u_{t+2},\ldots,u_{t+s-1}$, such that $u_{t+j+1}\in N_Z(v_{i_j})$ for $1\leq j\leq s-2$. Considering that $C$ is a clique, we can verify that 
$v_{t+1}P_su_{i_{s-1}+1}u_{t+s}\cdots u_{n-k-t+1}u_1P_1v_{i_1}u_{t+2}u_{i_1+1}P_2v_{i_{2}}u_{t+3}\cdots P_{s-2}v_{i_{s-2}}u_{t+s-1}$ $u_{i_{s-2}+1}P_{s-1}v_{i_{s-1}}$ 
is exactly a Hamilton path in $G'[C\cup X]$ that starts from $v_{t+1}$ and ends at $v_{i_{s-1}}$. Thus we may conclude that $G'[C\cup X]$ always has a Hamilton path, say $P$.  Observe that there are exactly $k-2$ vertices in $V(H)\setminus X=V(G')\setminus (C\cup X)$, and all of them have at least $t$ neighbors in $C$. By connecting each of these vertices to an arbitrary neighbor on the Hamilton path $P$, we immediately obtain a spanning $k$-ended tree in $G'$, contrary to the assumption.
Consequently, $|N(X)|=t$, and this proves the claim. \qed\vspace{3mm}

According to Claim \ref{claim::2}, $G' \cong K_t \nabla (K_{n-k-2t+1} \cup (k+t-1)K_1)$. Note that $K_t \nabla (K_{n-k-2t+1} \cup (k+t-1)K_1)$ is $t$-connected. By Lemma \ref{lem::1}, $K_t \nabla (K_{n-k-2t+1} \cup (k+t-1)K_1)$ does not have a spanning $k$-ended tree. Also, from $n \ge k^2+t k+t+1$ we obtain
$$
\begin{aligned}
e(K_t \nabla (K_{n-k-2t+1} \cup (k+t-1)K_1))&= {n-k-t+1 \choose 2}+(k+t-1)t\\
&\ge {n-k-t \choose 2}+(k+t-1)^2+k+t.
\end{aligned}
$$ 

This completes the proof.\qed\vspace{3mm}

In order to prove Theorem \ref{thm::7}, we need two additional lemmas.

 \begin{lem}(Hong, Shu and Fang\cite{HSF}; Nikiforov \cite{N})\label{lem::4}
   	Let $G$ be a graph of order $n$ with minimum degree $\delta(G)\geq k$. Then
    \begin{equation*}\label{equ::3}
	\rho(G)\leq \frac{k-1+\sqrt{(k+1)^2+4(2e(G)-nk)}}{2}.
\end{equation*}
     \end{lem}

Let $A=(a_{ij})_{n \times n}$ and $B=(b_{ij})_{n \times n}$ be two square matrices of order $n$ with real entries. Define $A \le B$ if  $a_{ij} \le b_{ij}$ for $1 \le i, j \le n$, and $A < B$ if $A \le B$ and $A \neq B$.

    \begin{lem}(Berman and Plemmons \cite{BP}; Horn and Johnson \cite{HJ})\label{lem::5}
	Let $O$ be an $n \times n$ zero matrix, $A=(a_{ij})$ and $B=(b_{ij})$ be two $n \times n$ matrices with spectral radius $\rho (A)$ and $\rho (B)$, respectively. If $O \le A \le B$, then $\rho (A) \le \rho (B)$. Furthermore, if $B$ is irreducible and $O \le A < B$, then $\rho (A) < \rho(B)$. 
    \end{lem}

Now we give the proof of Theorem  \ref{thm::7}.

{\flushleft \it Proof of Theorem \ref{thm::7}.}  Suppose that $G$ does not have a spanning $k$-ended tree. By Lemma \ref{lem::5},  
$$
\rho(G)\geq \rho(K_t \nabla (K_{n-k-2t+1} \cup (k+t-1)K_1)) > \rho(K_{n-k-t+1})=n-k-t.
$$
Since $G$ has minimum degree at least $t$, by  Lemma \ref{lem::4}, 
$$
\rho(G) \le \frac{t-1+\sqrt{(t+1)^2+4(2e(G)-nt)}}{2}.
$$ 
Combining the above two inequalities, we obtain
$$
\begin{aligned}
e(G) &> \frac{1}{2}(n^2-(2k+2t-1)n+k^2+2t^2+3kt-k-2t)\\
&= {n-k-t \choose 2}+(k+t-1)^2+k+t+n-k^2-\frac{3}{2}kt-\frac{1}{2}t^2-\frac{1}{2}t-1\\
&\ge {n-k-t \choose 2}+(k+t-1)^2+k+t,
\end{aligned}
$$ 
where the last inequality follows from  $n\geq k^2+\frac{3}{2}kt+\frac{1}{2}t^2+\frac{1}{2}t+1$. Let $G'=C_{n-1}(G)$. By Theorem \ref{thm::6}, we get $G' \cong K_t \nabla (K_{n-k-2t+1} \cup (k+t-1)K_1)$, and so  $\rho(G)\leq \rho(G')=\rho(K_t \nabla (K_{n-k-2t+1} \cup (k+t-1)K_1))$. Therefore, we conclude that $G=G'\cong \rho(K_t \nabla (K_{n-k-2t+1} \cup (k+t-1)K_1))$, and the result follows. 
   \qed

\section*{Acknowledgement}

X. Huang is supported by the National Natural Science Foundation of China (Grant No. 11901540).

\end{document}